\documentclass[11pt]{article} % arXiv Êýѧ³£¼û

\usepackage[margin=1in]{geometry}

\usepackage{microtype}

\usepackage{amsmath,amssymb,amsthm}

\usepackage{titlesec}
\titleformat{\section}{\normalfont\Large\bfseries}{\arabic{section}.}{1em}{}
\titleformat{\subsection}{\normalfont\large\bfseries}{\arabic{section}.\arabic{subsection}.}{1em}{}

\numberwithin{equation}{section}

\usepackage{hyperref}
\usepackage[nameinlink]{cleveref}
\hypersetup{
  colorlinks=false,    % Èô²»Ïë²ÊÉ«Á´½Ó£¬¸ÄΪ false
  linkcolor=blue,
  citecolor=blue,
  urlcolor=blue
}

\newtheorem{theorem}{Theorem}[section]
\newtheorem{lemma}[theorem]{Lemma}
\newtheorem{corollary}[theorem]{Corollary}

\theoremstyle{definition}

\theoremstyle{remark}

\crefname{theorem}{Theorem}{Theorems}
\crefname{lemma}{Lemma}{Lemmas}
\crefname{corollary}{Corollary}{Corollaries}
\crefname{proposition}{Proposition}{Propositions}
\crefname{definition}{Definition}{Definitions}
\crefname{remark}{Remark}{Remarks}
\crefname{equation}{Equation}{Equations}

\usepackage{indentfirst}        % Èà section/subsection ºóµÚÒ»¶ÎÒ²Ëõ½ø
\title{Your Title}
\author{Your Name}
\date{} % arXiv ³£Áô¿Õ

\title{Essential norm and Schatten class of Hankel operators between weighted Fock spaces}
\author{Yi Liu, Yufeng Lu}

\begin{document}
\maketitle
\begin{abstract}
In this paper, we characterize the essential norm of Hankel operators
from a weighted Fock space $F_{\varphi}^{p}$ to a weighted Lebesgue space $L_{\varphi}^{q}$
for all $1 \leq p, q < \infty$.
Additionally, we characterize the Schatten-$h$ class membership for Hankel operators acting from $F_{\varphi}^{2}$ to $L_{\varphi}^{2}$, where the class is defined by a continuous increasing convex function $h$.
\end{abstract}

\section{Introduction}

Consider the $n$-dimensional complex Euclidean space $\mathbb{C}^{n}$ endowed with the Lebesgue volume measure $v$.
Let $\varphi \in C^2(\mathbb{C}^n)$ be a real-valued weight function. Through the standard identification $\mathbb{C}^n \cong \mathbb{R}^{2n}$ given by
\(
z_j = x_{2j-1} + i x_{2j} \quad (j = 1, \dots, n),
\)
we view $\varphi$ as a $C^2$ function on $\mathbb{R}^{2n}$.
We say $\varphi$ is uniformly strictly convex if there exist constants $0 < m \leq M$
such that the real Hessian
satisfies
\begin{equation}\label{weight_varphi}
m I \leq\operatorname{Hess}_{\mathbb{R}} \varphi(x) = \left( \frac{\partial^2 \varphi(x)}{\partial x_j \partial x_k} \right)_{j,k=1}^{2n}
 \leq M I, \quad \forall x \in \mathbb{R}^{2n},
\end{equation}
where $I$ is the $2n \times 2n$ identity matrix.
 We define the weighted $L^p$-space $L^{p}_{\varphi}$. This space consists of all Lebesgue measurable functions $f$ on $\mathbb{C}^{n}$ with finite norm
\[
\|f\|_{p, \varphi} = \left( \int_{\mathbb{C}^{n}} |f(z) e^{-\varphi(z)}|^p  dv(z) \right)^{1/p}.
\]
The corresponding Fock space $F^{p}_{\varphi}$ is the closed subspace of holomorphic functions:
\[
 F^{p}_{\varphi} = L^{p}_{\varphi} \cap H(\mathbb{C}^{n}),
\]
where $H(\mathbb{C}^{n})$ denotes the space of entire functions on $\mathbb{C}^{n}$. When equipped with the norm $\|\cdot\|_{p, \varphi}$, $ F^{p}_{\varphi}$ forms a Banach space for $p \geq 1$. These spaces generalize the classical Fock spaces obtained when $\varphi(z) = (\alpha/2)|z|^2$ for $\alpha > 0$.

The Bergman projection $P: L^{2}_{\varphi} \to F^{2}_{\varphi}$ admits the integral representation
\[
P f(z) = \int_{\mathbb{C}^{n}} f(w) K(z, w) e^{-2 \varphi(w)} dv(w),
\]
which extends to a bounded operator from $L^{p}_{\varphi}$ to $ F^{p}_{\varphi}$ for all $1 \leq p \leq \infty$. This projection satisfies the reproducing property $P f = f$ for $f \in  F^{p}_{\varphi}$.

To define Hankel operators with unbounded symbols, we introduce a key dense subset of $F^{p}_{\varphi}$  for $1 \leq p < \infty$. It is well-established \cite{IDA} that the set of finite linear combinations of reproducing kernels, denoted by
\[
\Gamma = \left\{ \sum_{j=1}^{N} a_j K(\cdot, z_j) :
N \in \mathbb{N},\ a_j \in \mathbb{C},\ z_j \in \mathbb{C}^n \text{ for } 1 \leq j \leq N \right\}.
\]
We define the symbol class
\[
\mathcal{S} = \left\{ f \text{ measurable on } \mathbb{C}^n : f g \in L^{1}_{\varphi} \text{ for all } g \in \Gamma \right\}.
\]
For $f \in \mathcal{S}$, we consider the Hankel operator $H_f$ densely defined on $F^{p}_{\varphi}$ by
\[
H_f g(z) = (I - P)(fg)(z) = f(z)g(z) - P(fg)(z),
\]
where $P$ is the Bergman projection.

The following are the properties of the Bergman kernel $K(\cdot, \cdot)$, which are referred  in \cite{IDA}.

(1) There exist positive constants $\theta$ and $C_{1}$, depending only on $n, m$ and $M$, such that for all $z, w \in \mathbb{C}^{n}$,
\[
|K(z, w)| \leq C_{1} e^{\varphi(z)+\varphi(w)} e^{-\theta|z-w|}.
\]

(2) There exist positive constants $C_{2}$ and $r_{0}$ such that for all $z \in \mathbb{C}^{n}$ and $w \in B(z, r_{0})$,
\begin{equation}\label{eq:K_lower_estimate}
|K(z, w)| \geq C_{2} e^{\varphi(z)+\varphi(w)}.
\end{equation}

(3) For all $0<p<\infty$, the following norm equivalences hold for $z \in \mathbb{C}^{n}$:
\[
\frac{1}{C} e^{\varphi(z)} \leq \|K(\cdot, z)\|_{p, \varphi} \leq C e^{\varphi(z)}, \quad
\frac{1}{C} \leq \|k_{z}\|_{p, \varphi} \leq C,
\]
where $k_{z}(\cdot) = K(\cdot, z)/\sqrt{K(z, z)}$ is the normalized kernel. Additionally,
$
\lim_{|z| \to \infty} k_{z}(\xi) = 0
$
uniformly in $\xi$ on compact subsets of $\mathbb{C}^{n}$.

For a given radius \( r > 0 \), a sequence \( \{w_{j}\}_{j\in\mathbb{N}} \) in \( \mathbb{C}^{n} \) is said to form an \( r \)-lattice if it satisfies the following two conditions:
\begin{equation*}
 \quad \bigcup_{j} B(w_{j}, r) = \mathbb{C}^{n} \quad \text{and} \quad  \quad B\left(w_{j}, \tfrac{r}{2\sqrt{n}}\right) \cap B\left(w_{k}, \tfrac{r}{2\sqrt{n}}\right) = \emptyset \quad \text{for all } j \neq k.
\end{equation*}
We can easily see that there exists an integer $N$, depending only on the dimension of $\mathbb{C}^{n}$, such that for any $r$-lattice $\{a_{k}\}_{k=1}^{\infty}$, the following inequality holds:
\begin{equation*}
1 \leq \sum_{k=1}^{\infty} \chi_{B(a_{k}, 2r)}(z) \leq N \quad \text{for all} \  z \in \mathbb{C}^{n},
\end{equation*}
where $\chi_{E}$ denotes the characteristic function of a set $E \subset \mathbb{C}^{n}$.
Let \( w_{1} \in \mathbb{C}^{n} \) be a fixed base point and \( r > 0 \). We define the fundamental lattice:
\begin{equation} \label{eq:lambda-def}
\Lambda := \left\{w_{1} + \tfrac{r}{\sqrt{n}}(m + i s) : m, s \in \mathbb{Z}^{n}\right\}.
\end{equation}
This construction naturally yields an \( r \)-lattice structure on \( \mathbb{C}^{n} \). For any positive integer \( K \), we can identify the finite subset:
\begin{equation*}
\left\{w_{1} + (z_{1}, \ldots, z_{n}) \in \Lambda : 0 \leq \operatorname{Re} z_{j}, \operatorname{Im} z_{j} < K \tfrac{r}{\sqrt{n}} \text{ for } j=1,\ldots,n\right\} = \{w_{1}, \ldots, w_{K^{2n}}\}.
\end{equation*}
This allows us to define, for each \( 1 \leq k \leq K^{2n} \), the  sublattice:
\begin{equation}\label{eq:lambda_k-def}
\Lambda_{k} := \left\{w_{k} + K \tfrac{r}{\sqrt{n}}(m + i s) : m, s \in \mathbb{Z}^{n}\right\}.
\end{equation}
These sublattices exhibit the following key properties:
\begin{equation*}
 \quad \Lambda = \bigcup_{k=1}^{K^{2n}} \Lambda_{k} \quad \text{and} \quad  \quad \Lambda_{j} \cap \Lambda_{k} = \emptyset \text{ whenever } j \neq k.
\end{equation*}

For a function $f \in L_{\text{loc}}^{q}(\mathbb{C}^{n})$ (the space of all locally square integrable functions on $\mathbb{C}^{n}$), we define
\[
G_{q,r}(f)(z) = \inf_{h \in H(B(z, r))} \left( \frac{1}{|B(z, r)|} \int_{B(z, r)} |f - h|^{q} \, dv \right)^{\frac{1}{q}} \quad (z \in \mathbb{C}^{n})
\]
where$H(B(z, r))$ denotes the set of all holomorphic functions on
$
B(z, r) = \{ w \in \mathbb{C}^{n} : |z - w| < r \}
$
and $|B(z, r)| = \int_{B(z, r)} dv$ .

For $0 < s \leq \infty$, the space $\mathrm{IDA}^{s,q}$ (Integral Distance to Analytic Functions) consists of all $f \in L_{\mathrm{loc}}^{q}$ such that
\[
\|f\|_{\mathrm{IDA}^{s,q}} = \left\|G_{q,r}(f)\right\|_{L^{s}} < \infty
\]
for some $r > 0$. We write $\mathrm{BDA}^{q}$ for $\mathrm{IDA}^{\infty,q}$. The space $\mathrm{VDA}^{q}$ consists of all $f \in L_{\mathrm{loc}}^{q}$ such that $\lim\limits_{z \rightarrow \infty} G_{q,r}(f)(z) = 0$ for some $r > 0$.

For any  $z \in \mathbb{C}^{n}$ and function $f \in L^{q}(B(z, r), dv)$, where $r > 0$, the $q$-th mean oscillation of $f$ on the ball $B(z,r)$ is given by

\[
M_{q, r}(f)(z) = \left( \frac{1}{|B(z, r)|} \int_{B(z, r)} |f(w)|^{q} dv(w) \right)^{1/q}.
\]

Research on Hankel operators in classical Fock spaces traces back to early characterizations of compactness and boundedness. Bauer~\cite{Bauer2005} characterized Hankel operators with real-valued symbols on the classical Fock space. This foundational result was subsequently extended to generalized Fock spaces $F_{\alpha}^{p}$ by Per\"al\"a, Schuster and Virtanen~\cite{PSV2014}, and Hu and Wang \cite{Wang2018Hankelbetween} further generalized these results to mappings from $F_{\alpha}^{p}$ to $F_{\alpha}^{q}$.

For weighted Fock spaces $F_{\varphi}^{p}$ where $\varphi$ is subharmonic, real-valued, and satisfies $\operatorname{Hess}_{\mathbb{R}} \varphi \simeq \mathrm{E}$, Hu and Virtanen~\cite{Schatten,IDA} obtained comprehensive results. Their work established precise criteria for the boundedness and compactness of $H_f: F_{\varphi}^{p} \to L_{\varphi}^{q}$ ($1 \leq p, q < \infty$) and fully characterized the Schatten-$p$ class membership of $H_f: F_{\varphi}^{2} \to L_{\varphi}^{2}$ for $0 < p < \infty$.

Recent developments include the work of Lv and Wang~\cite{LvWang2024} on doubling Fock spaces, and Asghari, Hu, and Virtanen~\cite{AsghariHV2024} on Schatten class Hankel operators and the Berger-Coburn phenomenon. We observe that when  n=1 , the weight satisfying condition \eqref{weight_varphi} in our paper becomes a doubling measure, which was previously studied  in \cite{AsghariHV2024,LvWang2024}.

For Banach spaces \( X \) and \( Y \), the essential norm of a bounded linear operator \( T: X \rightarrow Y \) is defined by
\[
\|T\|_{\mathrm{ess}} = \inf \left\{ \|T-K\|_{X \rightarrow Y} : K \text{ compact from } X \text{ to } Y \right\},
\]
where \(\|\cdot\|_{X \rightarrow Y}\) stands for the operator norm from \(X\) to \(Y\). Lin and Rochberg \cite{Lin1993} gave an equivalent description of the essential norm for Hankel operators on Bergman spaces, which was later extended to weighted Bergman spaces by Asserda \cite{Asserda2006}. We establish analogous results for weighted Fock spaces to those obtained in \cite{Asserda2006,Lin1993}.

\begin{theorem}\label{TH:Hf_ess_Gq}
For $1 \leq p \leq q < \infty$, let \( f \in \mathcal{S} \) and $ H_{f}$ is bounded from $F_{\varphi}^{p}$ to $L_{\varphi}^{q}$. Then for any $r > 0$, the following quantities are equivalent:
\begin{enumerate}
    \item[(1)] \( \left\|H_{f}\right\|_{\text{ess}} \);

    \item[(2)] \( \displaystyle\limsup_{|z| \rightarrow \infty}\left\|H_{f}\left(k_{z}\right)\right\|_{L_{\varphi}^{q}} \);

    \item[(3)] \( \displaystyle\limsup _{|z| \rightarrow \infty}G_{q, r}(f)(z)  \);

    \item[(4)] \(\inf _{f=f_{1}+f_{2}}\left(\limsup _{|z| \rightarrow \infty}\left|\bar{\partial} f_{1}(z)\right|+\limsup _{|z| \rightarrow \infty} M_{q, r}\left(f_{2}\right)(z)\right)\), where the decomposition $f = f_1 + f_2$ is as provided by \autoref{thm:BDA-decomp}.
\end{enumerate}
\end{theorem}

\begin{theorem}\label{TH:Hf_ess_H}
For $1 \leq p \leq q < \infty$, suppose that $f \in \mathcal{S}$ and $H_{f}$ is bounded from $F_{\varphi}^{p}$ to $L_{\varphi}^{q}$.
Then we have
\[
\|H_{f}\|_{ess} \simeq \inf \left\{\|H_{f}-H\|_{F_{\varphi}^{p} \rightarrow L_{\varphi}^{q}}: H \text{ compact Hankel operator}\right\} \simeq \inf _{h \in \mathrm{VDA}^{q }}\|f-h\|_{\mathrm{BDA}^{q}}.
\]
\end{theorem}

This paper also investigates the Schatten-$h$ membership  for Hankel operators acting on weighted Fock spaces. The definition given in \cite[Definition 3]{ElFalla2016}.

\begin{theorem}\label{thm:Hf=Sh}
Suppose $h(\sqrt{(\cdot)}): \mathbb{R}^{+} \rightarrow \mathbb{R}^{+}$ is a continuous increasing convex function, suppose $f \in \mathcal{S}$ with $G_{2, r}(f)$ is bounded. Then the following conditions are equivalent:
\begin{enumerate}
\item[(1)] $H_{f} \in S_{h}$;
\item[(2)] there exists a constant $C>0$ such that for some (equivalent any) $r>0$,
\[
\int_{\mathbb{C}^n} h\left(C G_{2, r}(f)(z)\right) d v(z)<\infty.
\]
\end{enumerate}
\end{theorem}

The structure of this paper is organized as follows. Section 2 establishes the preliminary framework, first presenting the $\bar{\partial}$-equation and its solution operator, then developing the decomposition theory for $\mathrm{BDA}$ functions. Section 3 establishes the proof of our main result (\autoref{TH:Hf_ess_Gq} and \autoref{TH:Hf_ess_H}).  Finally, Section 4 investigates Schatten-$h$ class membership criteria for Hankel operators acting on the weighted Fock space $F^{2}_\varphi$.

In this paper, for $A, B>0$, the notation $A \lesssim B$ indicates that there exists some positive constant $C$ independent of the functions being considered such that $A \leq C B$ and $A \simeq B$ means that both $A \lesssim B$ and $B \lesssim A$ hold.

\section{Preliminaries}

\subsection{The $\bar{\partial}$-equation and the solution operator}

A pivotal tool in our analysis is the solution operator for the $\bar{\partial}$-equation. Given a weight function $\varphi$ satisfying \eqref{weight_varphi}, we define the operator $A_\varphi$ mapping $(0,1)$-forms to functions as in \cite{IDA}:
\begin{equation}\label{A_defenition}
A_\varphi(\omega)(z) = \int_{\mathbb{C}^n} e^{\langle 2\partial\varphi(\xi), z-\xi \rangle}
\sum_{j=0}^{n-1} \frac{
    \omega(\xi) \wedge \partial|\xi-z|^2 \wedge (2\bar{\partial}\partial\varphi(\xi))^j \wedge (\bar{\partial}\partial|\xi-z|^2)^{n-1-j}
}{j! |\xi-z|^{2n-2j}},
\end{equation}
where $\langle \partial\varphi(\xi), z-\xi \rangle = \sum_{k=1}^n \frac{\partial \varphi}{\partial \xi_k}(\xi)(z_k - \xi_k)$ .This operator provides a solution to the fundamental equation $\bar{\partial} u = \omega$. The properties of this operator are summarized in the following lemmas.

\begin{lemma}\cite[Lemma 2.4]{IDA}\label{lem:A_properties}
For $1 \leq p \leq \infty$, the following hold:
\begin{enumerate}
\item[(i)] There exists a constant $C > 0$ such that
\[
\| A_\varphi(\omega) \|_{p,\varphi} \leq C \| \omega \|_{p,\varphi}, \quad \forall \omega \in L_{0,1},
\]
where $L_{0,1}$ denotes the space of $(0,1)$-forms on $\mathbb{C}^n$.

\item[(ii)]  For $g \in \Gamma$ and $f \in C^2(\mathbb{C}^n)$ with $\bar{\partial}f \in L^p$,
\(
\bar{\partial} \left[ A_\varphi(g \bar{\partial}f) \right] = g \bar{\partial}f.
\)
\end{enumerate}
\end{lemma}

\begin{lemma}\cite[Corollary 2.5]{IDA}\label{lem:Hf_A}
Let $f \in \mathcal{S} \cap C^1(\mathbb{C}^n)$ with $\bar{\partial}f \in L^s$ ($1 \leq s \leq \infty$). Then for any $g \in \Gamma$,
\[
H_f(g) = A_\varphi(g \bar{\partial}f) - P\left(A_\varphi(g \bar{\partial}f)\right).
\]

\end{lemma}

\subsection{Decomposition of BDA Functions}

 Using the techniques developed by Hu and Virtanen \cite{IDA}, we establish the following decomposition method for characterizing symbol function.

Given $f \in L^q_{\text{loc}}$, we consider a $r$-lattice $\{a_{j}\}_{j=1}^{\infty}$ covering $\mathbb{C}^n$ with an associated partition of unity $\{\psi_j\}$. For each $j$, we construct local holomorphic approximants $h_j$ on $B(a_j, t)$ that minimize $M_{q,t}(f-h_j)(a_j)$. This leads to the decomposition:
\begin{equation}\label{eq:function_decop}
f = f_1 + f_2, \quad \text{where} \quad f_1 = \sum_{j=1}^\infty h_j \psi_j \quad \text{and} \quad f_2 = f - f_1.
\end{equation}

The crucial observation is that both $\bar{\partial} f_1$ and $f_2$ can be effectively controlled by the local oscillation of $f$. The following theorem provides a complete characterization of $\mathrm{BDA}^q$ and $\mathrm{VDA}^q$ spaces:

\begin{theorem}\cite[Corollary 3.9]{IDA}\label{thm:BDA-decomp}
Suppose $1 \leq q < \infty$, and $f \in L_{\mathrm{loc}}^{q}$. Then $f \in \operatorname{BDA}^{q}$ (or $\operatorname{VDA}^{q}$) if and only if $f$ admits a decomposition $f = f_{1} + f_{2}$ such that
\begin{align}\label{eq:3-26}
f_{1} &\in C^{2}(\mathbb{C}^{n}), \quad \bar{\partial} f_{1} \in L_{0,1}^{\infty} \quad \left(\text{or } \lim_{z \rightarrow \infty} |\bar{\partial} f_{1}| = 0\right)
\end{align}
and
\begin{align}
 M_{q, r}(f_{2}) &\in L^{\infty} \quad \left(\text{or } \lim_{z \rightarrow \infty} M_{q, r}(f_{2}) = 0\right) \label{eq:3-27}
\end{align}
for some (or any) $r > 0$. Furthermore,
\[
\|f\|_{\mathrm{BDA}^{q}} \simeq \inf \left\{ \left\|\bar{\partial} f_{1}\right\|_{L_{0,1}^{\infty}} + \left\|M_{q, r}(f_{2})\right\|_{L^{\infty}} \right\}
\]
where the infimum is taken over all possible decompositions $f = f_{1} + f_{2}$ that satisfy \eqref{eq:3-26} and \eqref{eq:3-27}.
\end{theorem}

Applying Theorem \ref{thm:BDA-decomp} to the case when $s=\infty$ (the $\text{BDA}^q$ space), Hu and Virtanen obtained a complete characterization for the boundedness and compactness of Hankel operators. This result serves as the foundation for our work.

\begin{theorem}\cite[Theorem 1.1(a)]{IDA}\label{thm:Hf_Bounded_Comp}
Let $f \in \mathcal{S}$ and suppose that  $ \varphi \in C^{2}\left(\mathbb{C}^{n}\right) $ is real valued and  $\operatorname{Hess}_{\mathbb{R}} \varphi \simeq \mathrm{E}$.
For $0 < p \leq q < \infty$ and $q \geq 1$, the Hankel operator $H_{f}:  F^{p}_{\varphi} \rightarrow L^{q}(\varphi)$ is bounded if and only if $f \in \mathrm{BDA}^{q}$, and $ H_{f}  $ is compact if and only if  $f \in \mathrm{VDA}^{q}$. Moreover, we have the norm estimate
\[
\left\|H_{f}\right\| \simeq \|f\|_{\mathrm{BDA}^{q}}.
\]
\end{theorem}

\section{Proofs of Theorems 1.1 and 1.2}

\begin{lemma}\label{le:f2_to_Hfess}
Let $1<p \leq q<\infty$, $r>0$, and $f \in \mathcal{S}$ . If
\[
\sup_{z \in \mathbb{C}^{n}} M_{q,r}(f)(z) < \infty,
\]
then
\[
\|H_{f}\|_{\mathrm{ess}} \lesssim \limsup_{|z| \rightarrow \infty} M_{q,r}(f)(z).
\]
\end{lemma}

\begin{proof}

For any $t>0$, let $\chi_{t}$ be the characteristic function of the set $\{z \in \mathbb{C}^{n} : |z| \leq t\}$. Let
$
f = \chi_{t}f + (1-\chi_{t})f
$
be the decomposition of $f$, then
\[
H_{f} = H_{\chi_{t}f} + H_{(1-\chi_{t})f}.
\]
Since $\chi_{t}f$ has compact support, hence $\chi_{t}f \in \mathrm{VDA}^q$ (since $G_{q,1}(\cdot) \to 0$ as $|z| \to \infty$). By \autoref{thm:Hf_Bounded_Comp}, $T_{1,t}$ is compact. Then
\begin{equation}\label{eq:Hfess}
\|H_{f}\|_{\mathrm{ess}} \leq \|H_{(1-\chi_{t})f}\|.
\end{equation}
Let $d\nu(z) = |(1-\chi_{t})f(z)|^{q} dv(z)$. By Fock-Carleson measure theory \cite[Theroem 2.6]{Lv2014Toeplitz}, we get
\[
\|H_{(1-\chi_{t})f}g\|_{q,\varphi}\lesssim \left\|\left(1-\chi_{t}\right) f g\right\|_{q,\varphi}\lesssim \|g\|_{ p,\varphi} \cdot \sup_{z \in \mathbb{C}^{n}} \left(\frac{1}{|B(z,r)|} \int_{B(z,r)} |(1-\chi_{t})f|^{q} dv\right)^{1/q}.
\]
Thus:
\begin{equation}\label{eq:Hf(1-t)}
\|H_{(1-\chi_{t})f}\| \lesssim \sup_{z \in \mathbb{C}^{n}} M_{q,r}((1-\chi_{t})f)(z)
\end{equation}
Let $t \rightarrow \infty$. When $|z|$ is sufficiently large, $\chi_{t} \equiv 0$ on $B(z,r)$, hence:
\begin{equation}\label{eq:lim}
\limsup_{|z| \rightarrow \infty} M_{q,r}((1-\chi_{t})f)(z) = \limsup_{|z| \rightarrow \infty} M_{q,r}(f)(z)
\end{equation}
Combining  \eqref{eq:Hfess}, \eqref{eq:Hf(1-t)} and \eqref{eq:lim}, we obtain
\[
\|H_{f}\|_{\mathrm{ess}} \lesssim \limsup_{|z| \rightarrow \infty} M_{q,r}(f)(z).
\]
\end{proof}

\begin{lemma}\label{le:f1_to_Hfess}
Let $1 \leq p \leq q < \infty$ and $f \in \mathcal{S}$. If \[\sup_{z \in \mathbb{C}^{n}}|\bar{\partial} f(z)| < \infty\] (in the weak derivative sense), then
\[
\|H_{f}\|_{ess} \lesssim \limsup_{|z| \rightarrow \infty} |\bar{\partial} f(z)|.
\]
\end{lemma}

\begin{proof}
According to  \autoref{lem:Hf_A} in the paper, for $g \in \Gamma$ , we have
\[
H_{f}g = A_{\varphi}(g\bar{\partial}f) - P\left(A_{\varphi}(g\bar{\partial}f)\right),
\]
where $A_{\varphi}$ is an  integral operator (defined in \eqref{A_defenition}) satisfying $L^{p}$ boundedness.

For $t > 0$, let $\chi_{t}$ be the characteristic function of the ball $B(0,t)$. Decompose $H_{f}$ as:
\[
H_{f}g = T_{1,t}g + T_{2,t}g,
\]
where
\[
\begin{aligned}
T_{1,t}g &= A_{\varphi}\left(\chi_{t}g\bar{\partial}f\right) - P\left(A_{\varphi}\left(\chi_{t}g\bar{\partial}f\right)\right), \\
T_{2,t}g &= A_{\varphi}\left((1-\chi_{t})g\bar{\partial}f\right) - P\left(A_{\varphi}\left((1-\chi_{t})g\bar{\partial}f\right)\right).
\end{aligned}
\]
Let $d\mu_{t}(z) = |\chi_{t}(z)\bar{\partial}f(z)|^{q}dv(z)$. By $\sup_{z}|\bar{\partial}f| < \infty$ and the compact support of $\chi_{t}$, $\mu_{t}$ is a compactly supported measure. Therefore, by Proposition 2.3 of \cite{IDA}, $\mu_{t}$ is a vanishing $(p,q)$-Fock-Carleson measure.
For any sequence $\{g_{n}\} \subset  F^{p}_{\varphi}$ weakly converging to $0$, we have
\[
\|T_{1,t}g_{n}\|_{q,\varphi} \lesssim \|\chi_{t}g_{n}\bar{\partial}f\|_{q,\varphi} = \left(\int |g_{n}|^{q}e^{-q\varphi}d\mu_{t}\right)^{1/q} \rightarrow 0 \quad (n \rightarrow \infty),
\]
since the embedding $ F^{p}_{\varphi} \hookrightarrow L^{q}(\mu_{t})$ is compact. Therefore, $T_{1,t}$ is compact.

Let $M_{t} = \sup_{|z| \geq t}|\bar{\partial}f(z)|$. Define the measure $d\nu_{t}(z) = |(1-\chi_{t}(z))\bar{\partial}f(z)|^{q}dv(z)$.
By \autoref{lem:A_properties} and the boundedness of the Bergman projection $P$:
\[
\|T_{2,t}g\|_{q,\varphi} \lesssim \|(1-\chi_{t})g\bar{\partial}f\|_{q,\varphi} = \left(\int |g|^{q}e^{-q\varphi}d\nu_{t}\right)^{1/q}.
\]
By the characterization of $(p,q)$-Fock-Carleson measures  \cite[Theroem 2.6]{Lv2014Toeplitz}:
\[
\sup_{z \in \mathbb{C}^{n}}\left(\frac{1}{|B(z,1)|}\int_{B(z,1)} d\nu_{t}\right)^{1/q} \leq \sup_{z \in \mathbb{C}^{n}}\left(\frac{1}{|B(z,1)|}\int_{B(z,1)} M_{t}^{q}dv\right)^{1/q} = M_{t}.
\]
Thus,
\[
\|T_{2,t}\|_{ F^{p}_{\varphi} \rightarrow L^{q}_{\varphi}} \lesssim M_{t}.
\]
By the definition of essential norm:
\[
\|H_{f}\|_{ess} \leq \|T_{2,t}\| \lesssim M_{t} = \sup_{|z| \geq t} |\bar{\partial}f(z)|.
\]
Taking $t \rightarrow \infty$, we obtain
\[
\|H_{f}\|_{ess} \lesssim \limsup_{t \rightarrow \infty} |\bar{\partial}f(z)| = \limsup_{|z| \rightarrow \infty} |\bar{\partial}f(z)|.
\]
\end{proof}

\begin{proof}[Proof of \autoref{TH:Hf_ess_Gq}]
First we prove (1) $\Rightarrow$ (2).Note that
\[
\left\|H_{f}-K\right\|_{F_{\varphi}^{p} \rightarrow L_{\varphi}^{q}}
\geq\left\|\left(H_{f}-K\right)\left(k_{z}\right)\right\|_{q, \varphi}
\geq \left\|H_{f}\left(k_{z}\right)\right\|_{q, \varphi}-\left\|K\left(k_{z}\right)\right\|_{q, \varphi}
\]
For any compact operator $K$, we have $\left\|K\left(k_{z}\right)\right\|_{F^{q}(\varphi)} \rightarrow 0$ as $|z| \rightarrow \infty$, since $k_{z} \rightarrow 0$ weakly as $|z| \rightarrow \infty$.Then
\[
 \left\|H_{f}\right\|_{ess} \geq \left\|H_{f}-K\right\|_{F_{\varphi}^{p} \rightarrow L_{\varphi}^{q}} \geq  \limsup _{|z| \rightarrow \infty} \left\|H_{f}\left(k_{z}\right)\right\|_{q, \varphi }
\]
Now we prove(2) $\Rightarrow$ (3), Note that for  $0< r_{0}$, \eqref{eq:K_lower_estimate} provides
\begin{equation*}
\begin{aligned}
\left\|H_{f}\left(k_{z}\right)\right\|_{q, \varphi}^{q}
&\geq \int_{B\left(z, r_{0}\right)}\left|f k_{z}-P\left(f k_{z}\right)\right|^{q} e^{-q \varphi} d v \\
&\geq C \frac{1}{\left|B\left(z, r_{0}\right)\right|} \int_{B\left(z, r_{0}\right)}\left|f-\frac{1}{k_{z}} P\left(f k_{z}\right)\right|^{q} d v \\
&\geq C G_{q, r_{0}}^{q}(f)(z)
\end{aligned}
\end{equation*}
Thus
\begin{align*}
\left\|\left(H_{f}-K\right)\left(k_{z}\right)\right\|_{q, \varphi}
&\geq \left\|H_{f}\left(k_{z}\right)\right\|_{q, \varphi}-\left\|K\left(k_{z}\right)\right\|_{q, \varphi} \\
&\geq C G_{q, 1}(f)(z)-\left\|K\left(k_{z}\right)\right\|_{q, \varphi}
\end{align*}
Taking the superior limit on both sides yields
\[\displaystyle\limsup_{|z| \rightarrow \infty}\left\|H_{f}\left(k_{z}\right)\right\|_{q, \varphi}\geq C \displaystyle\limsup _{|z| \rightarrow \infty}G_{q, r}(f)(z).\]
(3) $\Rightarrow $(4). By Lemma 3.6 in \cite{IDA}, we have
\begin{equation}\label{eq:f1_estimate_by_G}
f_{1} \in C^{2}(\mathbb{C}^{n}), \quad \left|\bar{\partial} f_{1}(z)\right| \lesssim G_{q, r}(f)(z), \quad z \in \mathbb{C}^n
\end{equation}
and
\begin{equation}\label{eq:f2_estimate_by_G}
M_{q, r}\left(f_{2}\right)(z) \lesssim G_{q, r}(f)(z), \quad z \in \mathbb{C}^n.
\end{equation}
Hence
\[
\limsup _{|z| \rightarrow \infty} \left|\bar{\partial} f_{1}(z)\right|+\limsup _{|z| \rightarrow \infty} M_{q, r}\left(f_{2}\right)(z) \leq \limsup _{|z| \rightarrow \infty} G_{q, r}(f)(z).
\]
(4) $\Rightarrow$ (1).
Let $f = f_1 + f_2$ with $f_1 \in C^2(\mathbb{C}^{n})$. Since $H_f$ is bounded from $F_{\varphi}^{p}$ to $L_{\varphi}^{q}$, by \eqref{eq:f1_estimate_by_G}, \eqref{eq:f2_estimate_by_G} and \autoref{thm:Hf_Bounded_Comp}, we have:
\begin{equation*}
\sup_{z \in \mathbb{C}^{n}} \left|\bar{\partial} f_1(z)\right| \lesssim \sup_{z \in \mathbb{C}^{n}} G_{q,r}(f)(z) < \infty
\end{equation*}
and
\begin{equation*}
\sup_{z \in \mathbb{C}^{n}} M_{q,r}(f_2)(z) \lesssim \sup_{z \in \mathbb{C}^{n}} G_{q,r}(f)(z) < \infty
\end{equation*}
Thus it suffices to prove the following two estimates for any decomposition $f = f_1 + f_2$:
\begin{equation}\label{eq:Hf1ess_upper_estimate_f1}
\left\|H_{f_1}\right\|_{ess} \lesssim \limsup_{|z| \rightarrow \infty} \left|\bar{\partial} f_1(z)\right|
\end{equation}
and
\begin{equation}\label{eq:Hf2ess_upper_estimate_MO_f2}
\left\|H_{f_2}\right\|_{ess} \lesssim \limsup_{|z| \rightarrow \infty} M_{q,r}(f_2)(z)
\end{equation}
By the proof of Theorem 1.1 in \cite{IDA}, $f_2 \in \mathcal{S}$. Since $f \in \mathcal{S}$, it follows that $f_1 \in \mathcal{S}$. Applying \autoref{le:f2_to_Hfess} and \autoref{le:f1_to_Hfess} , we conclude that \eqref{eq:Hf1ess_upper_estimate_f1} and \eqref{eq:Hf2ess_upper_estimate_MO_f2} hold.

Now for any decomposition $f = f_1 + f_2$, we have:
\[
\|H_f\|_{ess} \leq \|H_{f_1}\|_{ess} + \|H_{f_2}\|_{ess} \lesssim \left( \limsup_{|z| \to \infty}|\overline{\partial} f_1(z)| + \limsup_{|z| \to \infty} M_{q,r}(f_2)(z) \right).
\]
Since this holds for \textit{every} decomposition $f = f_1 + f_2$, we can take the infimum over all such decompositions:
\[
\|H_f\|_{ess} \lesssim \inf_{f=f_1+f_2} \left( \limsup_{|z| \to \infty}|\overline{\partial} f_1(z)| + \limsup_{|z| \to \infty} M_{q,r}(f_2)(z) \right).
\]
This completes the proof .

\end{proof}

\begin{proof}[Proof of \autoref{TH:Hf_ess_H}]
Let $\sigma_t \in C_0^\infty(\mathbb{C}^n)$ be a characteristic function satisfying:
\begin{itemize}
\item[(1)] $\sigma_t(z) = 1$ for $|z| \leq t$,
\item[(2)] $\sigma_t(z) = 0$ for $|z| \geq t+1$,
\item[(3)] $|\nabla \sigma_t(z)| \leq C$ uniformly in $t$.
\end{itemize}
We define the function $\psi_t$ by
\[
\psi_t(z) = A_\varphi(\sigma_t \bar{\partial} f_1)(z),
\]
where $A_\varphi$ is the  operator defined in \eqref{A_defenition}.

 From Proposition 10 of \cite{Berndtsson1982}, we have
\[
\bar{\partial} \psi_t = \sigma_t \bar{\partial} f_1.
\]
This implies:
\begin{itemize}
\item[(1)] For $|z| > t+1$, $\psi_t$ is holomorphic since $\sigma_t \equiv 0$,
\item[(2)]  For $|z| \leq t+1$, $\psi_t \in C^\infty(\mathbb{C}^n)$ because $\sigma_t \bar{\partial} f_1$ is smooth with compact support and $A_\varphi$ preserves smoothness.
\end{itemize}

Define the symbol
\[
h_t = \psi_t + \sigma_t f_2
\]
and consider the Hankel operator $K_t = H_{h_t}$. We verify its compactness.

For $H_{\sigma_t f_2}$, since $\sigma_t f_2$ has compact support, $H_{\sigma_t f_2}$ is compact. For $H_{\psi_t}$, when $|z| > t+1$, $H_{\psi_t} = 0$ because $\psi_t$ is holomorphic. When $|z| \leq t+1$, $\psi_t$ is smooth and bounded. Therefore, there is a constant \( t_{0} \) such that for all \( t<t_{0} \leq|z|<\infty \),
\[
G_{q, r}\left(\psi_{t}\right)(z)=0
\]
which implies that
\[
\lim _{|z| \rightarrow \infty} G_{q, r}\left(\psi_{t}\right)(z)=0
\]
By this and applying \autoref{thm:Hf_Bounded_Comp}, we see that \( H_{\psi_{t}} \) is compact from \( F_{\varphi}^{p} \) to \( L_{\varphi}^{q} \).

Thus $K_t = H_{\psi_t} + H_{\sigma_t f_2}$ is compact as the sum of compact operators.

 We decompose the difference:
\[
H_f - K_t = (H_{f_1} - H_{\psi_t}) + (H_{f_2} - H_{\sigma_t f_2}).
\]
Applying Fock-Carleson measure theory \cite[Theroem 2.6]{Lv2014Toeplitz}, we obtain:
\begin{equation*}
\|H_{(1-\sigma_t)f_2}\|_{F_\varphi^p \to L_\varphi^q} \lesssim \sup_{z \in \mathbb{C}^n} M_{q,r}((1-\sigma_t)f_2)(z) \leq \sup_{|z| > t} M_{q,r}(f_2)(z).
\end{equation*}
From \autoref{lem:A_properties}, we have the identity:
\begin{equation*}
\bar{\partial}(f_1 - \psi_t) = \bar{\partial} f_1-\sigma_{t} \bar{\partial} f_{1}=(1-\sigma_t)\bar{\partial} f_1 .
\end{equation*}
For arbitrary $g \in F_\varphi^p$, the operator acts as:
\begin{equation*}
H_{f_1 - \psi_t} g = A_\varphi(g(1-\sigma_t)\bar{\partial} f_1) - P(A_\varphi(g(1-\sigma_t)\bar{\partial} f_1)).
\end{equation*}
Using \autoref{lem:A_properties} along with the boundedness property of $P$, we establish:
\begin{equation*}
\|H_{f_1 - \psi_t}\|_{F_\varphi^p \to L_\varphi^q} \lesssim \sup_{z \in \mathbb{C}^n} |(1-\sigma_t)(z)\bar{\partial} f_1(z)| \leq \sup_{|z| > t} |\bar{\partial} f_1(z)|.
\end{equation*}
Then
\[
\|H_{f} - K_t\| \lesssim \sup_{|z| > t} |\bar{\partial} f_1(z)| + \sup_{|z| > t} M_{q,r}(f_2)(z).
\]
By taking $t \to \infty$ and the infimum over  decompositions of $f$, \autoref{TH:Hf_ess_Gq} implies that
\[
\inf_{t>0} \|H_f - K_t\| \lesssim \inf_{f=f_1+f_2} \left(\limsup_{|z| \to \infty} |\bar{\partial} f_1(z)| + \limsup_{|z| \to \infty} M_{q,r}(f_2)(z)\right) \lesssim \|H_f\|_{\text{ess}}.
\]
Since each $K_t$ is a compact Hankel operator, we have
\[
\inf_H \|H_f - H\| \leq \inf_t \|H_f - K_t\| \lesssim \|H_f\|_{\text{ess}},
\]
where the infimum is taken over all compact Hankel operators $H$. Combining this with the trivial inequality $\|H_f\|_{\text{ess}} \leq \inf_H \|H_f - H\|$ completes the proof of (1).

Now we prove (2). Since $H_{\psi}$ is compact if and only if $\psi \in V D A^{q}$, using (1) and \ref{thm:Hf_Bounded_Comp} we have
\begin{align*}
\left\|H_{f}\right\|_{\text{ess }} & \simeq \inf \left\{\left\|H_{f}-H\right\|: H \text{ compact Hankel operator }\right\} \\
& = \inf _{\psi \in V D A^q}\left\{\left\|H_{f}-H_{\psi}\right\|\right\} \\
& = \inf _{\psi \in V D A^q}\left\{\left\|H_{f-\psi}\right\|\right\} \\
& \simeq \inf _{\psi \in V D A^q}\|f-\psi\|_{B D A^q} . \\
\end{align*}
\end{proof}

\section{Proof of Theorem 1.3}

\begin{theorem}\cite[Theorem 1.9]{AHLT}\label{thm:toeplitzSh}
Suppose $h: \mathbb{R}^{+} \rightarrow \mathbb{R}^{+}$ is a continuous increasing convex function, $\mu$ is a positive Borel measure such that Toeplitz operator $T_{\mu}: F_{\varphi}^{2} \rightarrow F_{\varphi}^{2}$ is compact. Then $T_{\mu} \in S_{h}$ if and only if there is a constant $C>0$ such that
\[
\int_{\mathbb{C}} h(C \widetilde{\mu}(z)) d v(z)<\infty
\]
where
$
\widetilde{\mu}(z):=\int_{\mathbb{C}}\left|k_{z}(w)\right|^{2} e^{-2 \varphi(w)} d \mu(w).
$
\end{theorem}

\begin{proof}[Proof of \autoref{thm:Hf=Sh}]
 Let $\Lambda$ be an $r$-lattice as in \eqref{eq:lambda-def}, decompose $\Lambda = \cup_{k} \Lambda_{k}$ as in \eqref{eq:lambda_k-def}, and let $\{e_a : a \in \Lambda_k\}$ denote an orthonormal basis of $F^2_{\varphi}$. We define two linear operators:

\begin{equation*}\label{eq:T_B_def}
\begin{aligned}
T &= \sum_{a \in \Lambda} k_a \otimes e_a: L^2_{\varphi} \rightarrow F^2_{\varphi},\\
B &= \sum_{a \in \Lambda} g_a \otimes e_a: L^2_{\varphi} \rightarrow L^2_{\varphi},
\end{aligned}
\end{equation*}
where the functions $g_a$ are given by:

\[
g_a =
\begin{cases}
\frac{\chi_{B(a,r)} H_f(k_a)}{\|\chi_{B(a,r)} H_f(k_a)\|} & \text{if } \|\chi_{B(a,r)} H_f(k_a)\| \neq 0 ,\\
0 & \text{if } \|\chi_{B(a,r)} H_f(k_a)\| = 0.
\end{cases}
\]
The operator $B$ satisfies the following properties:
\begin{itemize}
\item [(a)]$\|g_a\| \leq 1$ for all $a \in \Lambda$,
\item [(b)]$\langle g_a, g_\tau \rangle = 0$ when $a \neq \tau$.
\end{itemize}
As a consequence, we obtain the norm estimate:
\[
\|B\|_{L^2_{\varphi} \rightarrow L^2_{\varphi}} \leq 1.
\]
On the other hand ,by the proof of Theorem 1.1 in \cite{Schatten},
\[\|T\|_{L^{2}_{\varphi} \rightarrow F^{2}_{\varphi}} \leq C.\]
Since
\begin{align*}
\langle B^* M_{\chi_{B(a,r)}} H_f T e_a, e_a \rangle &= \langle \chi_{B(a,r)} H_f T(e_a), B(e_a) \rangle
= \|\chi_{B(a,r)} H_f(k_a)\|
\end{align*}
and
\[
s_{j}\left(B^{*} H_{f} T\right) \leq\left\|B^{*}\right\| s_{j}\left(H_{f}\right)\|T\|\leq Cs_{j}\left(H_{f}\right),
\]
then

\begin{align*}
\sum_{a \in \Lambda_k} h\left(\left[C \int\limits_{B(a, r)}\left|H_{f}\left(k_{a}\right)\right|^{2} e^{-2 \varphi} d v\right]^{\frac{1}{2}}\right)
&= \sum_{a \in \Lambda_k} h\left(C\left\| \chi_{B(a, r)}H_{f} k_{a}\right\| \right) \\
&= \sum_{a \in \Lambda_k} h\left(C\left|\left\langle B^{*} H_{f} T e_a, e_a\right\rangle\right|\right) \\
&\leq \sum_{a \in \Lambda_k} h\left(C s_a\left(B^{*} H_{f} T\right)\right) \\
&\leq \sum_{a \in \Lambda_k} h\left(C s_a\left(H_{f}\right)\right) .
\end{align*}
Combining this with
\[\left( \int\limits_{B(a, r)}\left|H_{f}\left(k_{a}\right)\right|^{2} e^{-2 \varphi} d v\right)^{\frac{1}{2}}= \left(\int_{B(a, r)}\left|f-\frac{1}{k_{a}} P\left(f k_{a}\right)\right|^{2} d v\right)^{\frac{1}{2}}  \geq C G_{2,r}(f)(z)\]
then\[\sum_{a \in \Lambda_{k}} h(CG_{2,r}(f)(a)^{p})\leq \sum_{a \in \Lambda_{k}} h\left(C s_{k}\left(H_{f}\right)\right).\]

Now take $\Lambda$ to be an $\frac{r}{2}$-lattice similar to \eqref{eq:lambda-def}, which can be viewed as a union of $4^{n}  r$-lattice. Then

\begin{align*}
\int_{\mathbb{C}^{n}}h\left(C G_{2,\frac{r}{2}}(f)\right) d v & \leq \sum_{a \in \Lambda} \int_{B\left(a, \frac{r}{2}\right)} h\left(C G_{2,\frac{r}{2}}(f)\right) d v \\
& \leq C \sum_{a \in \Lambda} h\left(CG_{2,\frac{r}{2}}(f)(a)\right) \\
& \leq C \sum_{a \in \Lambda} h\left(CG_{2,r}(f)(a)\right) \leq C\sum_{k=1}^{\infty} h\left(C s_{k}\left(H_{f}\right)\right),
\end{align*}
and so, for $0<r \leq r_{0}$, we have
\[
\int_{\mathbb{C}^{n}}h\left(C G_{2,\frac{r}{2}}(f)\right) d v \leq C\sum_{k=1}^{\infty} h\left(C s_{k}\left(H_{f}\right)\right).
\]
Therefore, by Theorem 3.8 in \cite{IDA}, for each $r>0$, it holds that

\[
\int_{\mathbb{C}^{n}} h\left(CG_{2,r}(f)(a)\right) d v \leq C\sum_{k=1}^{\infty} h\left(C s_{k}\left(H_{f}\right)\right).
\]

(2) $\Rightarrow$ (1) Suppose $\int_{\mathbb{C}^n} h\left(C G_{2,r}(f)(z)\right) d v(z)<\infty$, $f$ admits a decomposition $f=f_{1}+f_{2}$ as in \eqref{eq:function_decop}. Using Lemma 3.6 in \cite{IDA}, we immediately obtain
\[
f_{1} \in C^{2}(\mathbb{C}^n), \quad \left|\bar{\partial} f_{1}(z)\right| \lesssim G_{2,r}(f)(z), \quad
 M_{2,r}\left(\bar{\partial} f_{1}\right)(z) \lesssim G_{2,r}(f)(z),\quad z \in\mathbb{C}^n.
\]
and
\[M_{2,r}\left(f_{2}\right)(z) \lesssim G_{2,r}(f)(z), \quad  \quad z \in \mathbb{C}^n. \]
 Hence
\begin{equation}\label{Sh1}
\int_{\mathbb{C}^n} h\left(M_{2,r}\left(\bar{\partial} f_{1}\right)\right) d v(z) \leq \int_{\mathbb{C}^n} h\left(C G_{2,r}(f)\right) d v(z)<\infty
\end{equation}
and
\begin{equation}\label{Sh2}
\int_{\mathbb{C}^n} h\left(M_{2,r}\left(f_{2}\right)\right) d v(z) \leq \int_{\mathbb{C}^n} h\left(C G_{2,r}(f)\right) d v(z)<\infty.
\end{equation}
Let $g$ be $\bar{\partial} f_{1}$ or $f_{2}$, and $M_{g}$ be the multiplication operator. We get $M_{2,r}(g)$ is bounded when $M_{2,r}(f)$ is bounded. So, $M_{g}$ is bounded from $F^{2}_{\varphi}$ to $L^{2}_{\varphi}$. Notice that
\[
\left\langle M_{g}^{*} M_{g} l, m\right\rangle=\left\langle M_{g} l, M_{g} m\right\rangle=\int_{\mathbb{C}^n} l \bar{m}|g|^{2} d v=\left\langle T_{|g|^{2}} l, m\right\rangle, \quad l, m \in F^{2}_{\varphi}.
\]
Hence \( M_{g}^{*} M_{g}=T_{|g|^{2}} \). Since \( M_{g} \in S_{h} \) if and only if \( T_{|g|^{2}} \in S_{h(\sqrt{(\cdot)})} \). By \autoref{thm:toeplitzSh} and the convexity of \( h(\sqrt{(\cdot)})\), \( T_{|g|^{2}} \in S_{h(\sqrt{(\cdot)})} \) if and only if
\[
\int_{\mathbb{C}^n} h\left(C\left(\widetilde{|g|^{2}}(z)^{\frac{1}{2}}\right)\right) d v(z)<\infty.
\]
Let \( \mu \) be a positive Borel measure, it is easy to check that \( \widehat{\mu}_{r}(z) \leq \widetilde{\mu}(z) \), and we claim that
\[
\int_{\mathbb{C}^n} h(\widetilde{\mu}(z)) d v(z) \leq \int_{\mathbb{C}^n} h\left(C \widehat{\mu}_{r}(z)\right) d v(z).
\]
By Jensen's inequality, the convexity of \( h \) and \( \widetilde{\mu}(z) \lesssim \widetilde{\widehat{\mu_{r}}}(z) \), we obtain
\[
\begin{aligned}
\int_{\mathbb{C}^n} h(\widetilde{\mu}(z)) d v(z) & \leq \int_{\mathbb{C}^n} h\left(\widetilde{\widehat{\mu}_{r}}(z)\right) d v(z) \\
& \leq \int_{\mathbb{C}^n} h\left(C \int_{\mathbb{C}^n}\left|k_{z,2}(w) e^{-\varphi(w)}\right|^{2} \widehat{\mu}_{r}(w) d A(w)\right) d v(z) \\
& \leq \int_{\mathbb{C}^n} \int_{\mathbb{C}^n}\left|k_{z,2}(w) e^{-\varphi(w)}\right|^{2} h\left(C \widehat{\mu}_{r}(w)\right) d A(w) d v(z) \\
& \lesssim \int_{\mathbb{C}^n} h\left(C \widehat{\mu}_{r}(z)\right) d v(z)
\end{aligned}
\]
It follows that
\[
\begin{aligned}
\int_{\mathbb{C}^n} h\left(C\left(\widetilde{|g|_{r}^{2}}(z)^{\frac{1}{2}}\right)\right) d v(z) & \leq \int_{\mathbb{C}^n} h\left(C \widehat{|g|_{r}^{2}}(z)^{\frac{1}{2}}\right) d v(z) \\
& =\int_{\mathbb{C}^n} h\left(C M_{2,r}(g)(z)\right) d v(z)\\
&\leq \int_{\mathbb{C}^n} h\left(C G_{2,r}(f)(z)\right) d v(z)<\infty
\end{aligned}\]
where the last inequality comes from \eqref{Sh1} and \eqref{Sh2}. Hence \( M_{g} \in S_{h} \). It is easy to get that \( \left\|H_{f_{1}} g\right\|_{L^{2}} \lesssim \left\|g \bar{\partial} f_{1}\right\|_{L^{2}} \) and \( \left\|H_{f_{2}} g\right\|_{L^{2}} \lesssim\left\|g f_{2}\right\|_{L^{2}} \) by using the proof process of Lemma 4.2 in \cite{Schatten}. Therefore, we get \( H_{f_{1}}, H_{f_{2}} \in S_{h} \). So, \( H_{f} \in S_{h} \). We complete the proof.

\end{proof}

\begin{corollary}
Let $1<p<\infty$ and suppose $f \in \mathcal{S}$ with $G_{2, r}(f)$ is bounded.Then the following statements are equivalent:
\begin{enumerate}
\item[(1)] $H_{f}: F^{2}_{\varphi} \rightarrow L^{2}_{\varphi}$ is in $S_{p}$.
\item[(2)] $f \in \mathrm{IDA}^{p}$.
\end{enumerate}
Furthermore,
\[
\left\|H_{f}\right\|_{S_{p}} \simeq \|f\|_{\mathrm{IDA}^{p}}.
\]
\end{corollary}

\renewcommand{\refname}{References}

\end{document}